\documentclass[12pt]{amsart}

\usepackage[margin=1.45in]{geometry}
\usepackage{amssymb}
\usepackage{amsmath}
\usepackage{amsthm}
\usepackage{cite}
\usepackage{enumitem}
\usepackage{cases}
\usepackage[font={it}, margin=1cm]{caption}
\usepackage{verbatim}
\usepackage{tikz}
\usetikzlibrary{decorations.markings, decorations.pathreplacing, positioning,arrows, matrix, decorations.pathmorphing, shapes.geometric, calc}
\usetikzlibrary{backgrounds, decorations}

\newtheorem{theorem}{Theorem}[section]
\newtheorem{lemma}[theorem]{Lemma} 
\newtheorem{proposition}[theorem]{Proposition} 

\newtheorem{question}{Question}

\newcommand{\p}[1]{\noindent {\newline\bf #1.}}
\newcommand{\out}{\operatorname{Out}}
\newcommand{\aut}{\operatorname{Aut}}
\newcommand{\inn}{\operatorname{Inn}}
\newcommand{\fix}{\operatorname{Fix}}
\newcommand{\im}{\operatorname{im}}
\newcommand{\BS}{\operatorname{BS}}
\newcommand{\bGamma}{\boldsymbol{\Gamma}}

\renewcommand{\>}{\rangle}

\title[Bass--Jiang group for aut.-induced HNN-extensions]{The Bass--Jiang group for automorphism-induced HNN-extensions}
\author{Alan D. Logan}
\address{School of Mathematical and Computer Sciences\\ Heriot-Watt University\\ Edinburgh, EH14 4AS, UK.}
\email{A.Logan@hw.ac.uk}



\begin{document}
\maketitle

\begin{abstract}
We use the Bass--Jiang group for automorphism-induced HNN-extensions to build a framework for the construction of tractable groups with pathological outer automorphism groups. We apply this framework to a strong form of a question of Bumagin--Wise on the outer automorphism groups of finitely presented, residually finite groups.
\end{abstract}

\section{Introduction}
\label{introduction}
Matumoto proved that every group $Q$ can be realised as the outer automorphism group of some group $G_Q$ \cite{matumoto1989any}. Many authors have refined this result by placing restrictions on the groups $Q$ and $G_Q$ \cite{kojima1988isometry} \cite{gobel2000outer} \cite{droste2001all} \cite{braun2003outer} \cite{BumaginWise2005} \cite{frigerio2005countable} \cite{minasyan2009groups} \cite{logan2015outer} \cite{LoganNonRecursive} \cite{LoganHNN}. Bumagin--Wise asked the following question \cite[Problem 1]{BumaginWise2005}:
\begin{question}[Bumagin--Wise]\label{q:BW}
Can every countable group $Q$ be realised as the outer automorphism group of a finitely generated, residually finite group $G_Q$?
\end{question}

Our main result is Theorem \ref{thm:BWfinitelypresented}, which relates to a strong form of Question~\ref{q:BW} where the group $G_Q$ is taken to be finitely presented rather than finitely generated. Almost nothing is known in general about the outer automorphism groups of finitely presented, residually finite groups.
A group $H$ has \emph{Serre's property FA} if every action of $H$ on any tree has a global fixed point;
this is related to the property of not splitting non-trivially as a free product with amalgamation or HNN-extension \cite[Theorem 15]{trees}.

\begin{theorem}
\label{thm:BWfinitelypresented}
Suppose $Q$ is a finitely presented group with Serre's property FA. Then there exists a finitely presented group $G_Q$ such that $Q$ embeds with finite index into $\out(G_Q)$. Moreover, if $Q$ is residually finite then $G_Q$ can be chosen to be residually finite.
\end{theorem}

Examples of finitely presented groups $Q$ with Serre's property FA are $\operatorname{SL}_3(\mathbb{Z})$ \cite{serre1974amalgames}, triangle groups \cite{trees}, finitely presented groups with Kazhdan's property T \cite{yasuo1982propertyT}, random groups in the sense of Gromov \cite{dahmani2011random}, R. Thompson's groups $T$ and $V$ \cite{farley2011proof}, and the Brin--Thompson groups $nV$ \cite{kato2015higher}.

The proof of Theorem \ref{thm:BWfinitelypresented} has two main ingredients. Firstly, Theorem \ref{thm:BWfinitelypresented} applies a certain framework built in this paper and explained below. Secondly, Theorem \ref{thm:BWfinitelypresented} applies a version of Rips' construction \cite{belegradek2008rips}. Rips' construction is usually applied to obtain finitely generated, non-finitely presentable groups with pathological properties; our application to finitely presented groups is unusual.

\p{The framework}
This paper principally consists of an in-depth analysis of the Bass--Jiang group for automorphism-induced HNN-extensions. The Bass--Jiang group of an HNN-extension $G$, defined in Section~\ref{sec:BJdefinitions}, is a subgroup of $\out(G)$ which arises naturally in the context of Bass--Serre theory.
Automorphism-induced HNN-extensions, defined in Section \ref{sec:AutHNN}, are a class of tractable HNN-extensions; by ``tractable'' we mean that they are an easy class of groups to work with and possess nicer properties than general HNN-extensions.
In Section \ref{sec:applications} we package the technical results of this analysis into a framework for the construction of tractable groups with pathological outer automorphism groups. This framework is applied to prove Theorem \ref{thm:BWfinitelypresented} (see also \cite{LoganNonRecursive} \cite{LoganHNN}).

\p{Describing \boldmath{$\out(G)$}}
Theorem \ref{thm:introthm} is a striking, non-technical special case of the framework described in Section \ref{sec:applications}. This theorem gives a convenient description of $\out(G)$ for certain residually finite HNN-extensions $G$. As our motivation comes from Question \ref{q:BW}, the assumption that $G$ is residually finite is completely natural. For functions $\rho$ and $\sigma$ write $\rho\sigma(x):=\sigma(\rho(x))$. Write $\gamma_g$ for the inner automorphism of $H$ corresponding to conjugation by $g$, so $\gamma_g(h):=g^{-1}hg$ for all $h\in H$.
For $\phi\in\aut(H)$ define $\fix(\phi):=\{h\in H: \phi(h)=h\}$.
Also define:
\[
A_{(K, \phi)}:=\{\delta\in\aut(H): \delta(K)=K, \:\exists\: a\in H \textnormal{ s.t. } \delta\phi(k)=\phi\delta\gamma_a(k) \:\forall\: k\in K\}
\]
and this is a subgroup of $\aut(H)$.

\begin{theorem}
\label{thm:introthm}
Let $G\cong \langle H, t; tkt^{-1}=\phi(k), k\in K\rangle$, $K\lneq H$ and $\phi\in\aut(H)$. Assume that $Z(H)\leq\fix(\phi)$, that $H$ is finitely generated and has Serre's property FA, and that $G$ is residually finite.
Then we have a short exact sequence:
\[
1\rightarrow C_H(K)\rtimes\frac{N_H(K)}{Z(H)K} \rightarrow \out^0(G)\rightarrow \frac{A_{(K, \phi)}\inn(H)}{\inn(H)}\rightarrow 1
\]
where $\out^0(G)$ is an index-one or -two subgroup of $\out(G)$.
\end{theorem}

\begin{proof}
Theorem \ref{thm:introthm} follows immediately from Theorems \ref{thm:equality} and \ref{thm:trivialcenter}.
\end{proof}
Theorem \ref{thm:trivialcenter} also classifies when $\out^0(G)=\out(G)$; these conditions are applicable here but are omitted for brevity. Note that the proof of Theorem \ref{thm:BWfinitelypresented} requires a more general result than Theorem \ref{thm:introthm}, as the groups $G=G_Q$ in Theorem \ref{thm:BWfinitelypresented} are not necessarily residually finite.

\p{A note on labels of results}
In a previous paper \cite{LoganNonRecursive} we cited certain results from this current paper. The labels of the results in this current paper have subsequently changed. We record the relevant changes here:
Theorem A is now Theorem \ref{thm:equality},
Lemma 2.1 is now Lemma \ref{lem:HhasFA},
Lemma 5.2 is now Lemma \ref{lem:finitecenter}, and
Proposition 5.3 is now Proposition \ref{prop:Aut}.

\p{Outline of the paper}
In Section~\ref{sec:AutHNN} we define automorphism-induced HNN-extensions and explain their ``tractability''.
In Section~\ref{sec:BJdefinitions} we define the Bass--Jiang group and the Pettet group for HNN-extensions.
In Section~\ref{sec:provingEquality} we prove conditions implying that the Bass--Jiang group is the whole outer automorphism group; our main result is Theorem \ref{thm:equality}.
In Section~\ref{sec:description} we give a description of the Bass--Jiang group for automorphism-induced HNN-extensions; our main result is Theorem \ref{thm:description}.
In Section~\ref{sec:commutativity} we prove results relating the structure of the Bass--Jiang group of an automorphism-induced HNN-extension $G$ to commutativity in the base group of $G$; our main result is Theorem \ref{thm:trivialcenter}.
In Section~\ref{sec:applications} we package the technical results of this paper into a framework for the construction of tractable groups with pathological outer automorphism groups, and we prove Theorem \ref{thm:BWfinitelypresented}.

\p{Acknowledgements} The author would like to thank Stephen J. Pride and Tara Brendle for many helpful discussions about this paper. He would also like to thank Henry Wilton for introducing him to the Belegradek--Osin version of Rips' construction, and the referee for their helpful comments and suggestions.


\section{Automorphism-induced HNN-extensions}
\label{sec:AutHNN}
Let $H$ be a group and let $\eta: K\rightarrow K^{\prime}$ be an isomorphism of subgroups of $H$. An \emph{HNN-extension of $H$ over $\eta$} is a group with relative presentation $G=\langle H, t; tkt^{-1}=\eta(k), k\in K\rangle$. We write $G=H\ast_{(K, K^{\prime}, \eta)}$.
An \emph{automorphism-induced HNN-extension} is an HNN-extension with relative presentation
\[
G=\<H, t; tkt^{-1}=\phi(k), k\in K\>
\]
where $\phi\in\aut(H)$ and $K\lneq H$;%
\footnote{If $H=K$ then $G=H\rtimes_{\phi}\mathbb{Z}$ is the mapping torus of $\phi$. Here the Bass--Jiang group (see Section \ref{sec:BJdefinitions}) is very different because $\phi$ extends to an inner automorphism \cite{logan2015outer}.}
the isomorphism of associated subgroups is induced by the automorphism $\phi$. We write $G=H\ast_{(K, \phi)}$. Throughout this paper we use the letters $G$, $H$, $K$ and $\phi$ as above.

In a general HNN-extension the form of the embeddings of the two associated subgroups $K$ and $K^{\prime}$ may be completely different. For example, $K$ may be normal while $K^{\prime}$ is malnormal. In contrast, in an automorphism-induced HNN-extension the form of the embeddings of $K$ and $K^{\prime}=\phi(K)$ are necessarily the same (both normal, both malnormal, and so on), and in practice the image group $\phi(K)$ may be disregarded. For example, supposing $H$ is finitely generated and residually finite, then $G=H\ast_{(K, \phi)}$ is residually finite if and only if $K$ is residually separable in $H$ \cite[Theorem A]{Logan2017Residual}. In this current paper, Theorems \ref{thm:introthm}, \ref{thm:description} and \ref{thm:trivialcenter} do not mention the image group $\phi(K)$.

The \emph{Baumslag--Solitar groups} $\BS(m, n)=\langle a, t; ta^{m}t^{-1}=a^{n}\rangle$ are the HNN-extensions of the infinite cyclic group with non-trivial associated subgroups. These provide standard examples of HNN-extensions with pathological properties. For example, Baumslag--Solitar groups can be non-Hopfian (they are Hopfian if and only if $n=1$, or $m=1$, or $\gcd(m, n)\neq1$ \cite{collins1983automorphisms}); they can be Hopfian but non-residually finite (they are residually finite if and only if $n=1$, or $m=1$, or $|m|=|n|$ \cite{meskin1972nonresidually}, and note that a finitely generated, residually finite group is Hopfian); and they can have non-finitely generated outer automorphism group (for example, $\out(\BS(2, 4))$ is not finitely generated \cite{collins1983automorphisms}). However, automorphism-induced Baumslag--Solitar groups have $|m|=|n|$, so are residually finite, by the above classification, and have virtually cyclic outer automorphism group \cite{GHMR}.

The above two paragraphs demonstrate that automorphism-induced HNN-extensions are more tractable and in certain cases have nicer properties than general HNN-extensions. Therefore, increasing the complexity of the base group $H$ maintains the tractability of these HNN-extensions but can allow for pathological properties; see, for example, \cite{LoganNonRecursive} where $H$ is a cubulated hyperbolic group with Serre's property FA, and \cite{LoganHNN} where $H$ is a triangle group.

\p{Notation}
We write $h^g:=g^{-1}hg$, and recall that we similarly write $\gamma_g(h):=g^{-1}hg$ for all $h\in G$. These contrast with the notation $tkt^{-1}=\phi(k)$ used here for HNN-extensions because it makes certain proofs more readable.

\p{The normaliser-quotient \boldmath{$N_H(K)/K$}}
The following proposition underlies this whole paper. It implies that $N_H(K)/K$ embeds into $\aut(G)$ for $G=H\ast_{(K, \phi)}$, and this paper represents an effort to obtain conditions implying that $N_H(K)/K$ is ``close to'' $\out(G)$ (for example, they are isomorphic, commensurable, and so on). We proved this proposition in collaboration with Ate\c{s} and Pride (see \cite{AtesPride}).

\begin{proposition}\label{prop:Aut}
Let $G=H\ast_{(K, \phi)}$. For $d\in N_H(K)$ define $\varphi_d: h\mapsto h$, $t\mapsto \phi(d)td^{-1}$. Then the map
\begin{align*}
\tau: N_H(K)&\mapsto \aut(G)\\
d&\mapsto \varphi_d
\end{align*}
is a homomorphism with kernel $K$.
\end{proposition}

\begin{proof}
It is routine to prove that $\varphi_d$ is a homomorphism. Then $\varphi_d$ is invertible, with inverse $\varphi_{d^{-1}}$, and so $\varphi_d\in\aut(G)$.
Now, $\tau$ is a homomorphism as $\varphi_{d_1}\varphi_{d_2}=\varphi_{d_1d_2}$, and it has kernel $K$ as $\phi(d)td^{-1}=t$ if and only if $d^{-1}\in K$ by Britton's Lemma \cite[Section IV.2]{L-S}.
\end{proof}

In the framework built by this paper and explained in Section \ref{sec:applications}, the pathological properties of $\out(G)$, where $G=H\ast_{(K, \phi)}$, are inherited from the normaliser-quotient $N_H(K)/K$. In particular, by Proposition \ref{prop:Aut}, if $N_H(K)/K$ is not residually finite and $H$ is finitely generated then $G$ is not residually finite \cite{Baumslag} (and this is an ``if and only if'' statement if $N_H(K)$ has finite index in $H$ \cite[Corollary 1.2]{Logan2017Residual}). We see later, in Theorem \ref{thm:trivialcenter}, that if $Z(H)\leq K\cap \fix(\phi)$ then $N_H(K)/K$ also embeds into $\out(G)$, while Lemma \ref{lem:finitecenter} gives a finite-index version of this embedding when $Z(H)\not\leq K$. Thus the properties of $N_H(K)/K$ are in a certain sense bestowed upon $G$.

{Note that if $H$ is infinite cyclic (and hence $G$ is a Baumslag--Solitar group) then $N_H(K)/K$ is cyclic, and so bestows no pathological properties upon $G$.}


\section{The Bass--Jiang group}
\label{sec:BJdefinitions}
This paper analyses the Bass--Jiang group for automorphism induced HNN-extensions. The proof of Theorem \ref{thm:BWfinitelypresented} applies this analysis. In this section we define the Bass--Jiang group $\out_H(G)$ of an HNN-extension $G=H\ast_{(K, K^{\prime}, \eta)}$. We also define the Pettet group $\out^H(G)$, which is applied in one of our main technical results, Theorem \ref{thm:equality}. The Pettet group lies between the Bass--Jiang group and the full outer automorphism group:
\[\out_H(G)\leq\out^H(G)\leq\out(G).\]
Note that these groups may be defined for graphs of groups in general \cite{Bass1996automorphism} \cite{pettet1999automorphism}, but for brevity we only define them for HNN-extensions.

\p{Bass--Serre theory}
The definitions in this section and certain proofs in Section~\ref{sec:provingEquality} apply Bass--Serre theory. All relevant definitions and notation are taken from Serre's book \cite{trees}. If $G=H\ast_{(K, K^{\prime}, \eta)}$ then $G$ acts in a standard way on a tree $\Gamma$, called \emph{the Bass--Serre tree of the HNN-extension $H\ast_{(K, K^{\prime}, \eta)}$}. Vertex-stabilisers $G_v$ of $\Gamma$ are precisely the conjugates of the base group $H$, so  $G_v=H^g$ for some $g\in G$, while edge-stabilisers $G_e$ are precisely the conjugates of the associated group $K$, so  $G_e=K^g$  for some $g\in G$. For vertices $v, w\in V\Gamma$ we write $[v, w]$ for the geodesic connecting $v$ and $w$, and we define the length of $[v, w]$, denoted $|[v, w]|$, to be the number of edges in $[v, w]$.

\p{The Bass--Jiang group}
If $\psi\in\aut(G)$ then $\widehat{\psi}$ denotes the coset of $\out(G)$ containing $\psi$.
The \emph{Bass--Jiang group} $\out_H(G)$ of $G=H\ast_{(K, K^{\prime}, \eta)}$ is the subgroup of $\out(G)$ consisting of those $\widehat{\psi}\in\out(G)$ which have a representative $\psi$ such that $\psi(H)=H$ and $\psi(t)=gt^{\epsilon}$ for some $g\in H$, $\epsilon=\pm1$. It is helpful to have a geometric view of the Bass--Jiang group: If $\aut_H(G)$ is the full pre-image of $\out_H(G)$ in $\aut(G)$ then $\aut_H(G)$ acts on the Bass--Serre $\Gamma$ of the HNN-extension $G=H\ast_{(K, K^{\prime}, \eta)}$ such that the diagram in Figure~\ref{fig:CommDiagBassJiang} commutes.

The Bass--Jiang group $\out_{\bGamma}(G_{\bGamma})$ for the fundamental group $G_{\bGamma}$ of a graph of groups $\bGamma$ was described by Bass--Jiang \cite{Bass1996automorphism}. Levitt described a group related to $\out_{\bGamma}(G_{\bGamma})$ in his investigation of $\out(G)$ for $G$ a one-ended hyperbolic group \cite{levitt2005automorphisms}. Gilbert--Howie--Metaftsis--Raptis gave conditions implying $\out_{\boldsymbol{\Gamma}}(G_{\bGamma})=\out(G_{\bGamma})$ for $G_{\bGamma}$ a {generalised Baumslag--Solitar group} \cite{GHMR}. Theorem \ref{thm:equality} mirrors this result of Gilbert--Howie--Metaftsis--Raptis.

\p{The Pettet group}
The \emph{Pettet group} $\out^H(G)$ of $G=H\ast_{(K, K^{\prime}, \eta)}$ is the subgroup of $\out(G)$ consisting of those $\widehat{\psi}\in\out(G)$ which have a representative $\psi$ such that $\psi(H)=H$. It is helpful to have a geometric view of the Pettet group: If $\aut^H(G)$ is the full pre-image of $\out^H(G)$ in $\aut(G)$ then $\aut^H(G)$ acts on the Bass--Serre tree $\Gamma$ of the HNN-extension $G=H\ast_{(K, K^{\prime}, \eta)}$ such that the diagram in Figure~\ref{fig:CommDiagPettet} commutes.

M. Pettet studied $\aut^H(G)$, and implicitly the Pettet group \cite{pettet1999automorphism}. He focused on conditions for the Pettet group to equal the Bass--Jiang group; these are related to Conditions (\ref{item:PettetEqualBJCondition1}) and (\ref{item:PettetEqualBJCondition2}) from Theorem \ref{thm:equality}.

\begin{figure}[ht]
\begin{centering}
\begin{minipage}{0.494\textwidth}
\begin{tikzpicture}[>=angle 90]
\matrix(a)[matrix of math nodes, row sep=3em, column sep=2.5em, text height=1.5ex, text depth=0.25ex] {G& \aut_H(G)\\ \aut(\Gamma)\\};
\path[->](a-1-1) edge node[auto] {$\theta$}(a-1-2);
\path[->](a-1-2) edge (a-2-1);
\path[->](a-1-1) edge (a-2-1);
\end{tikzpicture}
\caption{\footnotesize This diagram commutes, where $\aut_H(G)$ is the full pre-image of the Bass--Jiang group, $G=H\ast_{(K, K^{\prime}, \eta)}$, $\Gamma$ is the Bass--Serre tree, and $\theta$ is the canonical homomorphism with $\theta(G)=\inn(G)$.}\label{fig:CommDiagBassJiang}
\end{minipage}
\begin{minipage}{0.494\textwidth}
\begin{tikzpicture}[>=angle 90]
\matrix(a)[matrix of math nodes, row sep=3em, column sep=2.5em, text height=1.5ex, text depth=0.25ex] {G& \aut^H(G)\\ \aut(V\Gamma)\\};
\path[->](a-1-1) edge node[auto] {$\theta$}(a-1-2);
\path[->](a-1-2) edge (a-2-1);
\path[->](a-1-1) edge (a-2-1);
\end{tikzpicture}
\caption{\footnotesize This diagram commutes, where $\aut^H(G)$ is the full pre-image of the Pettet group, $G=H\ast_{(K, K^{\prime}, \eta)}$, $V\Gamma$ is the vertices of the Bass--Serre tree, and $\theta$ is the canonical homomorphism with $\theta(G)=\inn(G)$.}\label{fig:CommDiagPettet}
\end{minipage}
\end{centering}
\end{figure}

\p{Two examples}
We now give two examples which demonstrate that, in general, $\out_H(G)\lneq \out^H(G)$ and $\out^H(G)\lneq\out(G)$.
Our first example is essentially due to Collins--Levin \cite{collins1983automorphisms} (see also Levitt \cite{levitt2007GBSautomorphism}). Consider the HNN-extension $G_1=\langle a, t; ta^2t^{-1}=a^4\rangle$, with base group $H_1=\langle a\rangle$. The map $\psi_1: a\mapsto a, t\mapsto t^{-1}a^2t^{2}$ is an automorphism of $G_1$ \cite[Lemma 1.6]{collins1983automorphisms}. Clearly $\widehat{\psi}_1\in\out^{H_1}(G_1)$. However, it is a straight forward exercise in HNN-extensions, applying for example Britton's lemma, to prove that $\widehat{\psi}_1\not\in\out_{H_1}(G_1)$. Hence, $\out_{H_1}(G_1)\lneq \out^{H_1}(G_1)$.

We shall write $F(x, y)$ for the free group with free basis $\{x, y\}$.
Consider the automorphism-induced HNN-extension $G_2=\langle a, b, t; a^t=b\rangle$, with base group $H_2=F(a, b)$. Clearly $G_2=F(a, t)$, and hence the map $\psi_2: a\mapsto at, t\mapsto at^2$ is an automorphism of $G_2$. However, $\widehat{\psi}_2\not\in\out^{H_2}(G_2)$ as $at$ is not conjugate to any power of $a$ in $F(a, t)$, so $\out^{H_2}(G_2)\lneq\out(G_2)$.


\section{The Bass--Jiang group versus $\out(G)$}
\label{sec:provingEquality}
The Bass--Jiang group $\out_H(G)$ is most of interest when it is the full outer automorphism group $\out(G)$, so $\out_H(G)=\out(G)$. In this section we prove Theorem \ref{thm:equality}, which gives conditions implying that $\out_H(G)=\out(G)$. To prove these conditions we use the Pettet group $\out^H(G)$, as defined in Section~\ref{sec:BJdefinitions}. Throughout the remainder of this paper we assume all HNN-extensions are automorphism-induced, so $G=H\ast_{(K, \phi)}$, unless we explicitly state otherwise. We emphasize that this assumption is necessary for our results.

Lemma \ref{lem:HhasFA} gives a condition implying $\out^H(G)=\out(G)$. Lemma \ref{lem:MirrorPettet} gives conditions implying $\out_H(G)=\out^H(G)$. These lemmas combine to prove Theorem \ref{thm:equality}.

\subsection{Conditions implying \boldmath{$\out^H(G)=\out(G)$}}
Lemma \ref{lem:HhasFA} now proves that for $G=H\ast_{(K, \phi)}$, if $H$ has Serre's property FA then the Pettet group is equal to the full outer automorphism group. We then have $\out_H(G)\leq\out^H(G)=\out(G)$.

We note the similarity between Lemma \ref{lem:HhasFA} and a comment of Pettet \cite[Introduction]{pettet1999automorphism}. The principal difference is that Pettet is additionally assuming conditions implying $\out_H(G)=\out^H(G)$; we require no such preliminary assumptions. A subgroup $K$ of a group $H$ is \emph{conjugacy maximal} if there does not exist any $a\in H$ such that $K\lneq K^a$. The proof of Lemma \ref{lem:HhasFA} shows that under the conditions of the lemma the base group $H$ is always conjugacy maximal in $G=H\ast_{(K, \phi)}$. The proof of Lemma \ref{lem:HhasFA} may be written purely algebraically; we have used Bass--Serre theory so as to be similar to the other proofs on Section~\ref{sec:provingEquality}. Recall that $K\neq H$ in an automorphism-induced HNN-extension.

\begin{lemma}
\label{lem:HhasFA}
Let $G=H\ast_{(K, \phi)}$. If $H$ has Serre's property FA then $\out^H(G)=\out(G)$.
\end{lemma}

\begin{proof}
We prove that if $\psi\in\aut(G)$ then $\psi(H)=H^g$, where $g\in G$; this is sufficient as then $\psi\gamma_g^{-1}$ is a representative for $\widehat{\psi}$ such that $\psi\gamma_g^{-1}(H)=H$. So, let $\psi\in\aut(G)$ and consider the action of $\psi(H)$ on the Bass--Serre tree $\Gamma$ of $G$. As $H$ has Serre's property FA, $\psi(H)$ stabilises some vertex of $\Gamma$. Hence, $\psi(H)\leq H^{g_1}$ with $g_1\in G$. Similarly, $\psi^{-1}(H)\leq H^{g_2}$ with $g_2\in G$, and so $H^{g_3}\leq \psi(H)$ with $g_3:=\psi(g_2)^{-1}\in G$.

Suppose that $\psi(H) \lneq H^{g_1}$ and we shall find a contradiction. By the above, $H^{g_3}\lneq H^{g_1}$. Now, there exist vertices $u,w\in V\Gamma$ such that $G_u=H^{g_1}$ and $G_w=H^{g_3}$. If $u=w$ then $H^{g_1}=H^{g_3}$, a contradiction. So $u\neq w$ and $H^{g_3}$ stabilises the geodesic $[u, w]$. In particular, $H^{g_3}$ stabilises the initial edge of $[u, w]$, and therefore there exists $g_4\in G$ such that $H^{g_3}=H^{g_4}\leq K^{g_4}$ or $H^{g_3}=H^{g_4}\leq \phi(K)^{g_4}$.
Hence $H=K$, a contradiction.
\end{proof}

\subsection{Conditions implying \boldmath{$\out_H(G)=\out^H(G)$}}
\label{sec:provingEqualityPettet}
We now work towards proving Lemma \ref{lem:MirrorPettet}, which gives sufficient conditions for the Bass--Jiang group to be equal to the Pettet group of $G=H\ast_{(K, \phi)}$. We then have $\out_H(G)=\out^H(G)\leq\out(G)$.

\p{Conjugacy}
We begin our proof of Lemma \ref{lem:MirrorPettet} with Lemma \ref{lem:conjugationAut}, which describes conjugacy in automorphism-induced HNN-extensions. A word $U$ in an HNN-extension $G=H\ast_{(K, K^{\prime}, \eta)}$ is a product of the form $h_0t^{\epsilon_1}h_1t^{\epsilon_2}h_2\cdots h_{n-1}t^{\epsilon_n}h_n$ where $h_i\in H$, $\epsilon_i=\pm1$.
For words $U$ and $V$ we write $U\equiv V$ if $U$ and $V$ are exactly the same word, and we say $U$ is \emph{$t$-reduced} if it contains no subwords of the form $tkt^{-1}$, $k\in K$, or $t^{-1}k^{\prime}t$, $k^{\prime}\in K^{\prime}$. Every element of an HNN-extension has a representative word which is $t$-reduced, by Britton's Lemma.
\begin{lemma}\label{lem:conjugationAut}
Let $G=H\ast_{(K, \phi)}$. For all $g\in G$ there exists $i\in\mathbb{Z}$ and $a\in H$ such that if $b, c\in H$ with $b^g=c$ then $c=\phi^{i}\gamma_{a}(b)$.
\end{lemma}

Note that $i$ and $a$ are dependent on $g$, not on $b$ and $c$.

\begin{proof}
Let $g\in G$ be arbitrary. Let $u, w\in V\Gamma$ be the vertices of the Bass--Serre tree $\Gamma$ of $G$ such that $G_u=H$ and $G_w=H^g$. Let $W$ be a $t$-reduced word in $H\ast_{(K, \phi)}$ representing $g$. We induct on $n=|[u, w]|$, the length of the geodesic $[u, w]$. If $n=0$ then $u=w$ and the result trivially holds. For the induction step, note that if $n>0$ then $[u, w]$ decomposes as $[u, v]$ and $[v, w]$ where $|[u, v]|=n-1$ and $|[v, w]|=1$. We also have a decomposition $W\equiv Vd_{n-1}t^{\epsilon_n}d_n$, $d_i\in H$ and where $V$ is such that $G_v=H^V$ and the result holds for $V$. Then:
\begin{align*}
c=b^W
&=(b^V)^{d_{n-1}t^{\epsilon_n}d_n}\\
&=(\phi^j\gamma_{a_{n-1}}(b))^{d_{n-1}t^{\epsilon_n}d_n}\\
&=\phi^i\gamma_{a}(b)
\end{align*}
as required, where $i=j-\epsilon_n$ and $a=\phi^{-\epsilon_n}(a_{n-1}d_{n-1})d_n$.
\end{proof}

We now prove Lemma \ref{lem:LevEqualPet}, which corresponds to Condition (\ref{item:PettetEqualBJCondition3}) from Lemma \ref{lem:MirrorPettet} (and Condition (\ref{item:PettetEqualBJCondition3}) from Theorem \ref{thm:equality}).

\begin{lemma}
\label{lem:LevEqualPet}
Let $G=H\ast_{(K, \phi)}$. Suppose that $\out_H(G)\lneq\out^H(G)$. Additionally, suppose that there does not exist any $b\in H$ such that $\phi(K)\lneq \gamma_b(K)$ and there does not exist any $c\in H$ such that $\gamma_c(K)\lneq \phi(K)$. Then there exists $a\in H$ such that $\phi(K)=\gamma_{a}(K)$.
\end{lemma}

\begin{proof}
Suppose $\widehat{\psi}\in\out^H(G)\setminus\out_H(G)$, and let $\psi\in\widehat{\psi}$ be such that $\psi(H)=H$. Recall that $\psi$ acts on the vertices of the Bass--Serre tree $\Gamma$ of $G$. Let $u, v\in V\Gamma$ be connected by a single edge $e^{\prime}$, and suppose that $G_u=H$ and that $G_v=H^{t^{-1}}$. Then $G_{e^{\prime}}=K$. Write $w:=\psi(v)$, so $G_w=H^{\psi(t)^{-1}}$, and consider the geodesic $[u, w]$. As $\psi\not\in\aut_H(G)$, $|[u, w]|>1$. Note that $\psi(K)$ stabilises $[u, w]$ because $K$ stabilises $e^{\prime}$ (or, algebraically, because $\psi(K)\leq\psi(H), \psi(H^{t^{-1}})$), so $\psi(K)\leq G_{[u, w]}$. Note also that the first edge $e_1$ in this geodesic has stabiliser an $H$-conjugate either of $K$ or of $\phi(K)$, and without loss of generality we can assume $G_{e_1}$ is either $K$ or $\phi(K)$.

There exists a edge $e_j$ in the geodesic $[u, w]$ such that $G_{e_j}=\psi(K)$ \cite[Lemma 2.2]{pettet1999automorphism}, so $\psi(K)=K^g$ for some $g\in G$. Therefore, $G_{e_j}=K^g=\psi(K)\leq G_{[u, w]}\leq G_{e_1}$. As $G_{e_1}=K$ or $\phi(K)$, we have that $G_{e_j}=K^g\leq K$ or $\phi(K)$, so $G_{e_j}=\phi^i\gamma_{a_0}(K)$ for some $i\in\mathbb{Z}$, $a_0\in H$ by Lemma \ref{lem:conjugationAut}.

Recall that $[u, w]$ contains more than one edge. Let $f\in[u, w]$ be adjacent to $e_j$. Then $G_f=G_{e_j}^{a_1t^{\epsilon}a_2}$ for some $a_1, a_2\in H, \epsilon=\pm1$, so 
\[
\phi^i\gamma_{a_0}(K)=G_{e_j}\leq G_{[u, w]}\leq G_f=G_{e_j}^{a_1t^{\epsilon}a_2}=\phi^i\gamma_{a_0}(K)^{a_1t^{\epsilon}a_2}.
\]
Noting that $\phi^i\gamma_{a_0}(K)\leq H$, we therefore have that every element of $\phi^i\gamma_{a_0}(K)$ is equal to an element of the form $(a_1t^{\epsilon}a_2)^{-1}\phi^i\gamma_{a_0}(k)a_1t^{\epsilon}a_2$, $k\in K$, and hence of the form $\phi^{i-\epsilon}\gamma_{a_3}(k)$, with $i, \epsilon, a_3$ fixed. Hence, $\phi^i\gamma_{a_0}(K)\leq\phi^{i-\epsilon}\gamma_{a_3}(K)$, and so $\gamma_{a_4}(K)\leq\phi^{-\epsilon}(K)$. By the assumptions of the lemma $\gamma_{a}(K)=\phi(K)$, as required, where $a=a_4$ if $\epsilon=-1$ or $a=\phi(a_4)^{-1}$ if $\epsilon=1$.
\end{proof}

\p{Residual finiteness}
We now prove Lemma \ref{lem:BaumslagTretkoffGeneral}, which provides conditions for $\out_H(G)=\out^H(G)$ to hold in a general HNN-extension $G=H\ast_{(K, K^{\prime}, \eta)}$. Lemma \ref{lem:BaumslagTretkoffAutInd} is essentially Lemma \ref{lem:BaumslagTretkoffGeneral} for $G$ an automorphism-induced HNN-extension. If the conditions of either of these lemmas hold then the HNN-extension $G$ is residually finite.

Baumslag--Tretkoff gave conditions for general HNN-extensions $G=H\ast_{(K, K^{\prime}, \eta)}$ to be residually finite \cite{BaumslagTretkoff}, and the conditions of Lemma \ref{lem:BaumslagTretkoffGeneral} are a strong form of these residual finiteness conditions. A subgroup $N\leq H$ is \emph{characteristic} if $\psi(N)=N$ for all $\psi\in \aut(H)$.

\begin{lemma}
\label{lem:BaumslagTretkoffGeneral}
Let $G=\langle H, t; tkt^{-1}=\eta(k), k\in K\rangle$ where $\eta: K\rightarrow K^{\prime}$ is an isomorphism of subgroups of $H$, and let $H$ be finitely generated and residually finite. Suppose that for any positive integer $n$ and elements $x_i, y_i, z_i$ ($i=1, 2, \ldots, n$) with $x_i\in H\setminus K$, $y_i\in H\setminus K^{\prime}$, $z_i\in K\cap K^{\prime}\setminus 1$, there exists a characteristic subgroup $N$ of finite index in $H$ such that
\begin{enumerate}
\item $x_i K\cap N=y_i K^{\prime}\cap N=\emptyset$, the empty set, and $z_i\not\in N$,
\item $\eta$ maps $K\cap N$ into $N$ and $\eta^{-1}$ maps $K^{\prime}\cap N$ into $N$.
\end{enumerate}
Then $\out_H(G)=\out^H(G)$.
\end{lemma}

\begin{proof}
Note that $\eta$ induces an isomorphism $\eta^{\ast}$ of $KN/N$ onto $K^{\prime}N/N$. We write $\overline{G}=\langle H/N, t; t(kN)t^{-1}=\eta^{\ast}(kN), k\in K\rangle$. Then the natural map $h\mapsto hN$ induces a homomorphism of $G$ onto $\overline{G}$ \cite[Proof of Theorem 4.2]{BaumslagTretkoff}. Note that $H/N$ is finite, so $\out_{H/N}(\overline{G})=\out^{H/N}(\overline{G})$ \cite[Theorem 1]{pettet1999automorphism}.

Suppose $\out_H(G)\lneq \out^H(G)$. Then there exists $\mu\in\aut^H(G)$ with $\mu|_H\in\aut(H)$ such that $\mu(t)=h_1t^{\epsilon_1}h_2t^{\epsilon_2}\ldots h_mt^{\epsilon_m}$ is $t$-reduced with $m>1$.
Note that for each $j=1, \ldots, m$, either $h_j\in H\setminus K$ or $h_j\in H\setminus K^{\prime}$. Therefore, pick $n$ and the elements $x_i, y_i, z_i$, $i=1, \ldots, n$, such that each $h_j$ is an $x_i$ or $y_j$ as appropriate, pick $N$ characteristic in $H$ appropriately and form the corresponding group $\overline{G}$. Now, $\mu$ induces a map $\mu^{\ast}: \overline{G}\rightarrow\overline{G}$ with $\mu^{\ast}(t)=h_1Nt^{\epsilon_1}h_2Nt^{\epsilon_2}\ldots h_mNt^{\epsilon_m}$ and $\mu^{\ast}(h)=\mu(h)N$.

We now prove that $\mu^{\ast}\in\aut(\overline{G})$. Firstly, note that $\mu^{\ast}|_{H/N}\in\aut(H/N)$ as $N$ is characteristic in $H$. Then, $\mu^{\ast}(t)(kN)\mu^{\ast}(t)^{-1}=\eta^{\ast}(kN)$ for all $k\in K$ because $\mu$ and $\mu^{\ast}|_{H/N}$ are homomorphisms. Hence, $\mu^{\ast}:\overline{G}\rightarrow \overline{G}$ is a homomorphism. It is surjective because $\mu(t)$ and $H$ generate $G$, and it is injective because $\overline{G}$ is finitely generated and residually finite \cite[Theorem 3.1]{BaumslagTretkoff} and hence Hopfian.

Then $\mu^{\ast}\in\aut^H(\overline{G})\setminus\aut_H(\overline{G})$ as $\mu^{\ast}|_{H/N}\in\aut(H/N)$ and the word \[h_1Nt^{\epsilon_1}h_2Nt^{\epsilon_2}\ldots h_mNt^{\epsilon_m}\] is $t$-reduced (as each $h_j$ is an $x_i$ or $y_i$ as appropriate). Hence, $\out_{H/N}(\overline{G})\lneq\out^{H/N}(\overline{G})$, a contradiction.
\end{proof}

The only difference between the conditions of Lemma \ref{lem:BaumslagTretkoffGeneral} and those of Baumslag--Tretkoff is that we require the finite-index subgroup $N$ to be characteristic rather than normal. Suppose, as with Baumslag--Tretkoff, we instead assume that $N$ is normal in $H$. Then, as $H$ is finitely generated, there exists a characteristic subgroup $\widehat{N}\leq_fN$ which satisfies the first condition. However, it is not clear that $\widehat{N}$ is such that $\eta(K\cap\widehat{N})\leq \widehat{N}$.

Now, if $G=H\ast_{(K, \phi)}$ is automorphism-induced then the characteristic subgroup $\widehat{N}$ also satisfies the second condition, as for example $\phi(K\cap\widehat{N})=\phi(K)\cap\phi(\widehat{N})\leq \phi(\widehat{N})=\widehat{N}$. This observation gives the following lemma, which corresponds to Condition (\ref{item:PettetEqualBJCondition4}) of Lemma \ref{lem:MirrorPettet} (also compare with \cite[Lemma 4.4]{BaumslagTretkoff}). A subgroup $K$ of $H$ is \emph{residually separable} in $H$ if for all $x\in H\setminus K$ there exists a finite index, normal subgroup $N$ of $H$, written $N\lhd_f H$, such that $x\not\in KN$.

\begin{lemma}
\label{lem:BaumslagTretkoffAutInd}
Let $G=H\ast_{(K, \phi)}$, and let $H$ be finitely generated and residually finite. If $K$ is residually separable in $H$ then $\out_H(G)=\out^H(G)$.
\end{lemma}

\begin{proof}
As $K$ is residually separable in $H$ the conditions of Lemma \ref{lem:BaumslagTretkoffGeneral} hold. The result follows.
\end{proof}

An alternative proof of Lemma \ref{lem:BaumslagTretkoffAutInd} is as follows: By combining Lemma \ref{lem:LevEqualPet} with a result of Shirvani \cite[Lemma 3]{shirvani1985residually}, we see that the conditions of Lemma \ref{lem:BaumslagTretkoffAutInd} imply that $\phi(K)=\gamma_a(K)$ for some $a\in H$. The result follows from Condition (\ref{item:PettetEqualBJCondition2}) of Lemma \ref{lem:MirrorPettet}, below (the proof of this condition is independent of the proof of Condition (\ref{item:PettetEqualBJCondition4})).

\p{Proof of Lemma \ref{lem:MirrorPettet}}
Lemma \ref{lem:MirrorPettet} combines Lemma \ref{lem:LevEqualPet} and Lemma \ref{lem:BaumslagTretkoffAutInd} with results of Pettet.

\begin{lemma}
\label{lem:MirrorPettet}
Let $G=H\ast_{(K, \phi)}$.
Suppose that one of the following holds:
\begin{enumerate}
\item\label{item:PettetEqualBJCondition1} $K$ is conjugacy maximal in $H$;
\item\label{item:PettetEqualBJCondition2} $\phi(K)=\gamma_a(K)$ for some $a\in H$; or
\item\label{item:PettetEqualBJCondition3} There does not exist any $b\in H$ such that $\phi(K)\lneq \gamma_b(K)$ and there does not exist any $c\in H$ such that $\gamma_c(K)\lneq \phi(K)$; or
\item\label{item:PettetEqualBJCondition4} $H$ is finitely generated and residually finite, and $K$ is residually separable in $H$; or
\item\label{item:PettetEqualBJCondition5} $H$ is finitely generated and $G$ is residually finite.
\end{enumerate}
Then $\out_H(G)=\out^H(G)$.
\end{lemma}

\begin{proof}
If (\ref{item:PettetEqualBJCondition1}) or (\ref{item:PettetEqualBJCondition2}) holds then the result is known to hold  \cite[Lemma 2.6 and Theorem 1]{pettet1999automorphism}.
Now, note that it is impossible for (\ref{item:PettetEqualBJCondition3}) to hold simultaneously to the inequality $\out_H(G)\lneq\out(G)$: if both held then (\ref{item:PettetEqualBJCondition2}) would hold, by Lemma \ref{lem:LevEqualPet}, and so we obtain $\out_H(G)=\out(G)$, a contradiction. Hence, if (\ref{item:PettetEqualBJCondition3}) holds then the result holds.

If (\ref{item:PettetEqualBJCondition4}) holds then the result holds by Lemma \ref{lem:BaumslagTretkoffAutInd}. If (\ref{item:PettetEqualBJCondition5}) holds then (\ref{item:PettetEqualBJCondition4}) holds \cite[Theorem A]{Logan2017Residual}. The result follows.
\end{proof}

\subsection{The Bass--Jiang group versus \boldmath{$\out(G)$}}
Theorem \ref{thm:equality} now gives conditions implying that the Bass--Jiang group is the full outer automorphism group.

\begin{theorem}
\label{thm:equality}
Let $G\cong \langle H, t; tkt^{-1}=\phi(k), k\in K\rangle$, $K\lneq H$ and $\phi\in\aut(H)$, be an automorphism-induced $HNN$-extension.
If $H$ has Serre's property FA then $\out^H(G)=\out(G)$. If any one of the following statements holds then $\out_H(G)=\out^H(G)$:
\begin{enumerate}
\item\label{item:PettetEqualBJCondition1} $K$ is conjugacy maximal in $H$;
\item\label{item:PettetEqualBJCondition2} $\phi(K)=\gamma_a(K)$ for some $a\in H$; or
\item\label{item:PettetEqualBJCondition3} There does not exist any $b\in H$ such that $\phi(K)\lneq \gamma_b(K)$ and there does not exist any $c\in H$ such that $\gamma_c(K)\lneq \phi(K)$; or
\item\label{item:PettetEqualBJCondition4} $H$ is finitely generated and residually finite, and $K$ is residually separable in $H$; or
\item\label{item:PettetEqualBJCondition5} $H$ is finitely generated and $G$ is residually finite.
\end{enumerate}
\end{theorem}

\begin{proof}
Theorem \ref{thm:equality} follows immediately from Lemmas~\ref{lem:HhasFA}~and~\ref{lem:MirrorPettet}.
\end{proof}

Note that Theorem \ref{thm:equality} obtains $\out_H(G)=\out(G)$ in two disjoint steps; it proves that $\out^H(G)=\out(G)$ and separately that $\out_H(G)=\out^H(G)$. M. Pettet proved a similar result to Theorem \ref{thm:equality} involving Conditions~(\ref{item:PettetEqualBJCondition1}) and (\ref{item:PettetEqualBJCondition2}) \cite[Theorem 1]{pettet1999automorphism}, but he does not obtain these disjoint steps due to his more general setting.

Theorem \ref{thm:equality} gives a concrete method of proving that an automorphism-induced HNN-extension $G=H\ast_{(K, \phi)}$ is not residually finite: if $H$ is finitely generated and if there exists some $\psi\in\aut^H(G)\setminus\aut_H(G)$ (that is, $\out_H(G)\lneq\out^H(G)$) then $G$ is not residually finite. Lemma \ref{lem:BassJiangProperSubgp}, below, provides a way of proving this inequality.

\subsection{A condition implying \boldmath{$\out_H(G)\lneq\out^H(G)$}}
We do not know if the conditions of Lemma \ref{lem:MirrorPettet} are both necessary and sufficient for $\out_H(G)=\out^H(G)$ to hold, and reading closely the proof of Lemma \ref{lem:LevEqualPet} one might suspect that they are. Lemma \ref{lem:BassJiangProperSubgp} now demonstrates how close the conditions of Lemma \ref{lem:MirrorPettet} are to being necessary and sufficient.

Lemma \ref{lem:BassJiangProperSubgp} demonstrates this by proving that if Conditions~(\ref{item:PettetEqualBJCondition1})~and~(\ref{item:PettetEqualBJCondition3}) of Lemma \ref{lem:MirrorPettet} fail ``compatibly'' with one another then $\out_H(G)\lneq\out^H(G)$. Note that these two conditions fail for $G=H\ast_{(K, \phi)}$ if and only if, after replacing $\phi$ with $\phi^{-1}$ if necessary, there exist some $a, b\in H$ such that $\gamma_a(K)\lneq K\lneq \phi\gamma_b(K)$. The ``compatibility'' we require is that we can take $a\in \phi\gamma_b(K)$.

\begin{lemma}
\label{lem:BassJiangProperSubgp}
Let $G=H\ast_{(K, \phi)}$. Suppose also that $K^{a}\lneq K\lneq \phi\gamma_b(K)$ for some $a\in \phi\gamma_b(K)$, $b\in H$. Then $\out_H(G)\lneq\out^H(G)$.
\end{lemma}

\begin{proof}
By replacing $\phi$ with $\phi\gamma_b$, we may assume $b=1$.
Note that $a\not\in K$ as $\gamma_a(K)\lneq K$, and so the word $t^2at^{-1}\phi(a)^{-1}$ is $t$-reduced. We prove that the map
\begin{align*}\psi: h&\mapsto h&\forall h\in H\\
t&\mapsto t^2at^{-1}\phi(a)^{-1}\end{align*}
is an automorphism of $G$; then $\psi\in\aut^H(G)\setminus\aut_H(G)$ and the result follows.

Now, $\phi(a)^{-1}k\phi(a)\in\phi(K)$ so $\psi: G\rightarrow G$ is a homomorphism. To see that $\psi$ is a surjection we prove that $t\in\im(\psi)$. Begin with the following, where $a_1=\phi^{-1}(a)\in K$. Note that $a^{-1}a_1^{-1}a\in K$.
\begin{align*}
(t^2at^{-1}\phi(a)^{-1})a^{-1}
&=t^2at^{-1}\phi(a^{-1}a_1^{-1}aa^{-1})\\
&=t^2a_1^{-1}at^{-1}\phi(a^{-1})\\
&=ta^{-1}tat^{-1}\phi(a^{-1})
\end{align*}
We then have that:
\begin{align*}
&(t^2at^{-1}\phi(a)^{-1})\phi(a)\left[(t^2at^{-1}\phi(a)^{-1})a^{-1}\right](t^2at^{-1}\phi(a)^{-1})^{-1}\\
&=(t^2at^{-1}\phi(a)^{-1})\phi(a)\left[ta^{-1}tat^{-1}\phi(a)^{-1}\right](t^2at^{-1}\phi(a)^{-1})^{-1}\\
&=(t^3at^{-1}\phi(a)^{-1})(t^2at^{-1}\phi(a)^{-1})^{-1}\\
&=t
\end{align*}
Therefore, $\psi$ is surjective. It is injective because it is invertible, with inverse $\psi^{-1}: h\mapsto h$ for all $h\in H$, $t\mapsto t\phi(a)ta^{-1}t^{-1}$.
\end{proof}


\section{Describing the Bass--Jiang group}
\label{sec:description}
In this section we prove Theorem \ref{thm:description}, which gives a short-exact sequence description of the Bass--Jiang group of $G=H\ast_{(K, \phi)}$. Combined with Theorem \ref{thm:equality}, it allows us to give a description of $\out(G)$.

Bass--Jiang gave a decomposition of the group $\out_H(G)$ for $G$ a general HNN-extension \cite{Bass1996automorphism}. Their description took the form of a filtration. However, their description contains both of the associated subgroups $K$ and $\phi(K)$, while our description in Theorem \ref{thm:description} only contains $K$. The difference is in the description of the kernel $\out_H^{(V)}(G)$ from the short exact sequence in Theorem \ref{thm:description} ($\out_H^{(V)}(G)$ is formally defined below). Specifically, Bass--Jiang find some $N\unlhd\out_H^{(V)}(G)$ such that
\[
\frac{\out_H^{(V)}(G)}{N}\cong\frac{\left(N_H(\phi(K))\cap N_H(K)\right)/Z(H)}{Z(H)K}.
\]
In general, $N_H(\phi(K))\cap N_H(K)\neq N_H(K)$. Therefore Bass--Jiang's decomposition of $\out_H^{(V)}(G)$ is different from the decomposition in Theorem \ref{thm:description}. Our simpler description of $\out_H^{(V)}(G)$ is fundamental to the framework described in Section \ref{sec:applications} and is essential in proving Theorem \ref{thm:BWfinitelypresented}.

\p{Results of Bass--Jiang}
Our proof of Theorem \ref{thm:description} is based on two results of Bass--Jiang, which we state now. We have translated their results and definitions into our setting of automorphism-induced HNN-extensions. The first result is a classification of the elements of the Bass--Jiang group $\out_H(G)$ \cite[Theorem 5.6]{Bass1996automorphism}.

\begin{theorem}[Bass--Jiang]
\label{thm:principle}
Let $G=H\ast_{(K, \phi)}$. If ${\delta}\in \aut(H)$, $a\in H$ are such that ${\delta}(K)=K$ and ${\delta}\phi(k)=\phi{\delta}\gamma_a(k)$ for all $k\in K$ then the map
\begin{align*}\alpha_{(\delta, a)}: h&\mapsto {\delta}(h)&\forall h\in H\\
t&\mapsto at \end{align*}
is an automorphism of $G$.

If $\zeta\in\aut(H)$, $b\in H$ are such that $\zeta(K)=\phi(K)$, $\zeta^2\gamma_b(K)=K$, and $\zeta\phi^{-1}(k)=\phi\zeta\gamma_b(k)$ for all $k\in K$ then the map
\begin{align*}\beta_{(\zeta, b)}: h&\mapsto \zeta(h)&\forall h\in H\\
t&\mapsto bt^{-1} \end{align*}
is an automorphism of $G$.

Moreover, every element $\widehat{\psi}$ of $\out_H(G)$ has a representative in $\aut_H(G)$ of the form $\alpha_{(\delta, a)}$ or of the form $\beta_{(\zeta, b)}$.
\end{theorem}

The second result of Bass--Jiang which we apply requires certain definitions. The restrictions on $a$ and $b$ in the second, equivalent definition of $\out_H^{(V)}(G)$ follow from Theorem \ref{thm:principle}:
\begin{align*}
\out_H^0(G):=\{&\widehat{\alpha}_{(\delta, a)}: a\in H, {\delta}(K)=K, {\delta}\phi(k)=\phi{\delta}\gamma_a(k) \text{ for all }k\in K\},\\
\out_H^{(V)}(G):=\{&{\widehat{\alpha}_{(\gamma_b, a)}}: a\in H, \gamma_b\in\inn(H)\}\\
=\{&{\widehat{\alpha}_{(\gamma_b, a)}}: a\in H, b\in N_H(K), ba\phi(b)^{-1}\in C_H(\phi(K))\}.
\end{align*}
The subgroup $\out_H^0(G)$ clearly has index two in $\out_H(G)$ if there exists some automorphism $\beta_{(\zeta, b)}$ of $G$, and index one otherwise.

The second result of Bass--Jiang which we apply is the initial two steps of their filtration decomposition of $\out_H(G)$ \cite[Theorem 8.1\footnote{Section 8.5 of Bass--Jiang's paper is a restriction of their Theorem 8.1 to HNN-extensions. We cite their more general Theorem 8.1 because certain errors have crept in to the re-statement in Section 8.5.}]{Bass1996automorphism}.
Recall that the subgroup $A_{(K, \phi)}\leq\aut(H)$ was defined in the introduction.

\begin{theorem}[Bass--Jiang]
\label{thm:BassJiangSec85}
Let $G=H\ast_{(K, \phi)}$. The surjection
\begin{align*}
\chi_1:
\out_H^0(G)&\rightarrow \frac{A_{(K, \phi)}\inn(H)}{\inn(H)}\\
\widehat{\alpha}_{(\delta, a)}&\mapsto\widehat{\delta}
\end{align*}
has kernel $\out_H^{(V)}(G)$, and either $\out_H^0(G)=\out_H(G)$ or there exist some $\zeta\in\aut(H)$ and some $b\in H$ such that $\zeta(K)=\phi(K)$, $\zeta^2\gamma_b(K)=K$, and $\zeta\phi^{-1}(k)=\phi\zeta\gamma_b(k)$ for all $k\in K$, whence $\out_H^0(G)$ has index two in $\out_H(G)$.
\end{theorem}

\p{The subgroup \boldmath{$\out_H^{(V)}(G)$} of \boldmath{$\out_H^0(G)$}}
We wish to describe the kernel $\out_H^{(V)}(G)$ from Theorem \ref{thm:BassJiangSec85}; combining the description we obtain with Theorem \ref{thm:BassJiangSec85} proves the main result of this section, Theorem \ref{thm:description}. Our description is given by Lemmas \ref{lem:kernelsemidirect} and \ref{lem:BKisom}. These lemmas both mention the subgroup $C_K$ of $\out_H^{(V)}(G)$:
\begin{align*}
C_K
:=&\{{\widehat{\alpha}_{(1, a)}}: a\in H\}\\
=&\{{\widehat{\alpha}_{(1, a)}}: a\in C_H(\phi(K))\}
\end{align*}
Lemmas \ref{lem:kernelsemidirect} and \ref{lem:BKisom} also both apply the following technical result.

\begin{lemma}\label{lem:conjbyFk}
Let $G=H\ast_{(K, \phi)}$, and let $g\in G$ and $x, y \in H$. Suppose that both of the following hold.
\begin{align}
g^{-1}ytg&=x^{-1}\phi(x)t\label{equality:conjbyFk1}\\
g^{-1}hg&=x^{-1}hx &\forall \:h\in H\label{equality:conjbyFk2}
\end{align}
Then $g\in K$.
\end{lemma}

\begin{proof}
Begin by writing $g$ as a $t$-reduced word $zW$, where $z\in H$ and $W$ is either empty or begins with $t^{\pm1}$. Assume that $W$ is non-empty. Now, (\ref{equality:conjbyFk2}) has the form $W^{-1}z^{-1}hzW=x^{-1}hx$. As $K\lneq H$, $h$ can be chosen such that the word $W^{-1}z^{-1}hzW$ is $t$-reduced ($z^{-1}hz$ not in $K$ or $\phi(K)$ as appropriate), a contradiction by Britton's Lemma \cite[Section IV.2]{L-S}. Hence, $W$ is empty. So $g=z\in H$. It then follows from (\ref{equality:conjbyFk1}) that $g\in K$, as required.
\end{proof}

Our proofs of Lemmas~\ref{lem:kernelsemidirect}~and~\ref{lem:BKisom}, and of certain proofs in Section~\ref{sec:commutativity}, require calculations with automorphisms of the form $\alpha_{(\delta, a)}$.
We note the following identities:
\begin{align*}
\alpha_{(\delta_1, a_1)}\alpha_{(\delta_2, a_2)}&=\alpha_{(\delta_1\delta_2, \delta_2(a_1)a_2)}\\
\alpha_{(\delta, a)}^{-1}&=\alpha_{(\delta^{-1}, \delta^{-1}(a^{-1}))}
\end{align*}
Certain calculations are written in terms of cosets $\widehat{\alpha}_{(\delta, a)}$, but these follow immediately from the above two identities.

\begin{lemma}\label{lem:kernelsemidirect}
The map
\begin{align*}
\chi_2: \out_H^{(V)}(G)&\rightarrow N_H(K)/Z(H)K\\
\widehat{\alpha}_{(\gamma_b, a)}&\mapsto bZ(H)K
\end{align*}
is a well-defined surjective homomorphism with kernel $C_K$.
\end{lemma}

\begin{proof}
Suppose the map $\chi_2$ is well-defined. Note that for all $b\in N_H(K)$, the pair ${(\gamma_b, b^{-1}\phi(b))}$ defines the automorphism ${\alpha}_{(\gamma_b, b^{-1}\phi(b))}$ by Theorem \ref{thm:principle}, and so $\chi_2$ is surjective. As $\widehat{\alpha}_{(\gamma_{b_1}, a_1)}\widehat{\alpha}_{(\gamma_{b_2}, a_2)}=\widehat{\alpha}_{(\gamma_{b_1b_2}, a_3)}$ for some $a_3\in H$, we have that $\chi_2$ is also a homomorphism. Now, $C_K\leq \ker\chi_2$, but as $\widehat{\alpha}_{(\gamma_k, a)}=\widehat{\alpha}_{(1, ka\phi(k)^{-1})}$ for all $k\in K$ we have that $\ker\chi_2=C_K$.

To see that the map $\chi_2$ is well-defined, suppose $\widehat{\alpha}_{(\gamma_{b_1}, a_1)}=\widehat{\alpha}_{(\gamma_{b_2}, a_2)}$ and it is sufficient to prove that $b_1^{-1}b_2\in Z(H)K$. Note that if $\alpha_{(\gamma_b, a)}\in\aut(G)$ then $\alpha_{(\gamma_b, b^{-1}\phi(b))}\in\aut(G)$, and then ${\alpha}_{(\gamma_{b}, a)}={\alpha}_{(\gamma_{b}, b^{-1}\phi(b))}{\alpha}_{(1, \phi(b)^{-1}ba)}$. Therefore, the identity $\widehat{\alpha}_{(\gamma_{b_1}, a_1)}=\widehat{\alpha}_{(\gamma_{b_2}, a_2)}$ becomes:
\begin{align*}\widehat{\alpha}_{(\gamma_{b_1}, b_1^{-1}\phi(b_1))}\widehat{\alpha}_{(1, \phi(b_1)^{-1}b_1a_1)}&=\widehat{\alpha}_{(\gamma_{b_2}, b_2^{-1}\phi(b_2))}\widehat{\alpha}_{(1, \phi(b_2)^{-1}b_2a_2)}\\
\Rightarrow\widehat{\alpha}_{(1, a_3)}&=\widehat{\alpha}_{(\gamma_{b_1^{-1}b_2}, (b_1^{-1}b_2)^{-1}\phi(b_1^{-1}b_2))}
\end{align*}
for some $a_3\in H$.
Hence, $\alpha_{(1, a_3)}\gamma_g=\alpha_{(\gamma_{b_1^{-1}b_2}, (b_1^{-1}b_2)^{-1}\phi(b_1^{-1}b_2))}$ for some $g\in G$ and $a_3\in H$. Therefore, the conditions of Lemma \ref{lem:conjbyFk} hold for the two underlying identities (where $x:=b_1^{-1}b_2$ and $y:=a_3$), hence $g\in K$ with $g^{-1}hg=(b_1^{-1}b_2)^{-1}hb_1^{-1}b_2$ for all $h\in H$, and so $b_1^{-1}b_2\in Z(H)K$ as required.
\end{proof}

Note that $N_K:=\{\widehat{\alpha}_{(\gamma_b, b^{-1}\phi(b))}: b\in N_H(K)\}$ is a subgroup of $\out_H^{(V)}(G)$, and $\out_H^{(V)}(G)=C_KN_K$. So by Lemma \ref{lem:kernelsemidirect} we just need $C_K\cap N_K=1$ to obtain a semidirect product in the statement of Theorem \ref{thm:description}. This possible splitting of the map $\chi_2$ is discussed in Section~\ref{sec:commutativity}.

\begin{lemma}\label{lem:BKisom}
Define $L=\{\phi^{-1}(k)k^{-1}: k\in K\cap Z(H)\}$. Then $L\unlhd C_H(K)$ and $\displaystyle C_K\cong \frac{C_H(K)}{L}$.
\end{lemma}

\begin{proof}
We prove that $C_K\cong \frac{C_H(\phi(K))}{\phi(L)}$; the result then follows. To begin, note that $\widehat{\alpha}_{(1, a_1)}\widehat{\alpha}_{(1, a_2)}=\widehat{\alpha}_{(1, a_1a_2)}$, and so $C_K$ is a homomorphic image of $C_H(\phi(K))$. We therefore prove that $\alpha_{(1, a)}\in\inn(G)$ if and only if $a\in \phi(L)$.

Clearly if $a=k\phi(k)^{-1}\in \phi(L)$ then $\alpha_{(1, a)}=\gamma_{k}^{-1}\in\inn(G)$.
On the other hand, suppose $\alpha_{(1, a)}\in\inn(G)$. To see that $a\in \phi(L)$, note that the conditions of Lemma \ref{lem:conjbyFk} hold (where $x=1$, $y=a$, and in the following we take $k:=g^{-1}$).
Therefore, we have that $k^{-1}hk=h$ for all $h\in H$ and that $k^{-1}tk=at$, where $k\in K$. The first identity implies $k\in Z(H)$ (and so $k\in Z(H)\cap K$), while the second implies $a=k^{-1}\phi(k)$. Hence, $a\in \phi(L)$.
\end{proof}

\p{Describing \boldmath{$\out^H(G)$}}
We can now give our description of the Bass--Jiang group $\out_H(G)$ for $G=H\ast_{(K, \phi)}$.

\begin{theorem}\label{thm:description}
Let $G\cong \langle H, t; tkt^{-1}=\phi(k), k\in K\rangle$, $K\lneq H$ and $\phi\in\aut(H)$, be an automorphism-induced $HNN$-extension.
Let $L=\{\phi^{-1}(k)k^{-1}: k\in K\cap Z(H)\}$.
Then $L$ is a normal subgroup of $C_H(K)$ and we have a diagram of exact sequences:
\begin{center}
\begin{tikzpicture}
  \matrix (m) [matrix of math nodes, row sep=1em,
    column sep=1.5em]{
    & 1& & & \\
    & C_H(K)/L&&&\\
    1& \out_H^{(V)}(G)& \out_H^0(G) & \displaystyle\frac{A_{(K, \phi)}\inn(H)}{\inn(H)} & 1\\
    & N_H(K)/Z(H)K&&&\\
    & 1&&&\\};
  \path[-stealth]
    (m-1-2) edge (m-2-2)
    (m-2-2) edge (m-3-2)
    (m-3-2) edge (m-4-2)
    (m-4-2) edge (m-5-2)
    (m-3-1) edge (m-3-2)
    (m-3-2) edge (m-3-3)
    (m-3-3) edge (m-3-4)
    (m-3-4) edge (m-3-5)
;
\end{tikzpicture}
\end{center}
where either $\out_H^0(G)=\out_H(G)$ or there exist some $\zeta\in\aut(H)$ and some $a\in H$ such that $\zeta(K)=\phi(K)$, $\zeta^2\gamma_a(K)=K$ and $\zeta\phi^{-1}(k)=\phi\zeta\gamma_a(k)$ for all $k\in K$, whence $\out_H^0(G)$ has index two in $\out_H(G)$.
\end{theorem}

\begin{proof}
By Theorem \ref{thm:BassJiangSec85}, it is sufficient to prove that $\out_H^{(V)}(G)$ decomposes as in the statement of Theorem \ref{thm:description}. This follows from Lemmas~\ref{lem:kernelsemidirect}~and~\ref{lem:BKisom}.
\end{proof}

\p{An alternative view}
Recall the definition of $A_{(K, \phi)}$:
\[
A_{(K, \phi)}:=\{\delta\in\aut(H): \delta(K)=K, \:\exists\: a\in H \textnormal{ s.t. } \delta\phi(k)=\phi\delta\gamma_a(k) \:\forall\: k\in K\}
\]
The following alternative description of the image group can be substituted in to Theorem \ref{thm:description}. Here, $\inn(N_H(K))=\{\gamma_b: b\in N_H(K)\}$.
\[
\frac{A_{(K, \phi)}\inn(H)}{\inn(H)}
\cong
\frac{A_{(K, \phi)}}{\inn(N_H(K))}
\]
This isomorphism holds because $\inn(H)\cap A_{(K, \phi)}=\inn(N_H(K))$; that $\inn(H)\cap A_{(K, \phi)}$ is a subgroup of $\inn(N_H(K))$ is clear from the definition of $A_{(K, \phi)}$, while if $\gamma_b(K)=K$ then $\gamma_b\in A_{(K, \phi)}$ because $\gamma_b\phi(k)=\phi\gamma_b\gamma_{b^{-1}\phi(b)}(k)$.


\section{The impact of commutativity in $H$}
\label{sec:commutativity}
The reader might feel cheated that the decomposition of $\out_H^{(V)}(G)$ in Theorem \ref{thm:description} did not split over the subgroups $C_K$ and $N_K$, as defined in Section~\ref{sec:description}. Indeed, $\out_H^{(V)}(G)=C_KN_K$ and $C_K\unlhd \out_H^{(V)}(G)$. However, elements of the form $\widehat{\alpha}_{(\gamma_c, c^{-1}\phi(c))}$ where $c\in Z(H)$ are contained in both subgroups, so the decomposition of $\out_H^{(V)}(G)$ in Theorem \ref{thm:description} does not split in general.
In this section we prove Theorem \ref{thm:trivialcenter}, which assumes that $Z(H)\leq\fix(\phi)$; here the decomposition of $\out_H^{(V)}(G)$ does split over $C_K$ and $N_K$. We also prove related results which assume commutativity conditions on $H$.

\p{Splitting the description of Theorem \ref{thm:description}}
We begin by ``reversing'' the map $\chi_2$ from Lemma \ref{lem:kernelsemidirect}, which allows us to determine if $\chi_2$ splits.

\begin{lemma}\label{lem:kernelsemidirectreverse}
The map
\begin{align*}
\overline{\chi_2}: N_H(K)&\rightarrow\out_H^{(V)}(G)\\
b&\mapsto\widehat{\alpha}_{(\gamma_b, b^{-1}\phi(b))}
\end{align*}
is a homomorphism with kernel $JK$, where $J=Z(H)\cap \operatorname{Fix}(\phi)$.
\end{lemma}

\begin{proof}
As $\widehat{\alpha}_{(\gamma_{b_1}, b_1^{-1}\phi({b_1}))}\widehat{\alpha}_{(\gamma_{b_2}, {b_2}^{-1}\phi({b_2}))}=\widehat{\alpha}_{(\gamma_{{b_1}{b_2}}, {b_2}^{-1}{b_1}^{-1}\phi({b_1}{b_2}))}$, the map $\overline{\chi_2}$ is a homomorphism. Note that $\ker\overline{\chi_2}\geq JK$ as for $b\in J$, $k\in K$ we have ${\alpha}_{(\gamma_{bk}, (bk)^{-1}\phi(bk))}=\gamma_k$. To see that the $\ker\overline{\chi_2}\leq JK$, suppose that ${\alpha}_{(\gamma_b, b^{-1}\phi(b))}=\gamma_g$. Then the conditions of Lemma \ref{lem:conjbyFk} hold for the two underlying identities (where $x=b$, $y=1$), so $k:=g\in K$. Hence, $bk^{-1}\in Z(H)$ and $b^{-1}\phi(b)=k^{-1}\phi(k)$. The second identity implies that $bk^{-1}\in\operatorname{Fix}(\phi)$. Therefore, $bk^{-1}\in Z(H)\cap \operatorname{Fix}(\phi)=J$ so $b\in JK$ as required.
\end{proof}

The hypotheses of Theorem \ref{thm:trivialcenter} are those of Theorem \ref{thm:description} with the additional assumption that $Z(H)\leq\fix(\phi)$. Here the vertical exact sequence from Theorem \ref{thm:description} splits. Theorem \ref{thm:trivialcenter} does not immediately follow from Theorem \ref{thm:description}.
The subgroup $N_K$ corresponds to the factor $N_H(K)/Z(H)K$, while $C_K$ corresponds to the factor $C_H(K)$. The action of $N_K$ on $C_K$ is $\alpha_{(1, a)}^{\alpha_{(\gamma_b, b^{-1}\phi(b))}}=\alpha_{(1, \gamma_{\phi(b)}(a))}$.

\begin{theorem}\label{thm:trivialcenter}
Let $G$ be as in Theorem \ref{thm:description}, and additionally assume $Z(H)\leq\fix(\phi)$.
We have a short exact sequence:
\[
1\rightarrow C_H(K)\rtimes\frac{N_H(K)}{Z(H)K} \rightarrow \out_H^0(G)\rightarrow \frac{A_{(K, \phi)}\inn(H)}{\inn(H)}\rightarrow 1
\]
where the conditions for the index $|\out_H(G):\out_H^0(G)|$ are as in Theorem \ref{thm:description}.
\end{theorem}

\begin{proof}
As $Z(H)\leq\fix(\phi)$, Lemma \ref{lem:kernelsemidirectreverse} implies that the map $\chi_2$ from Lemma \ref{lem:kernelsemidirect} splits. Theorem \ref{thm:trivialcenter} then follows from Theorem \ref{thm:description}.
\end{proof}

\p{Virtually embedding \boldmath{$N_H(K)/K$} into \boldmath{$\out(G)$}}
Lemma \ref{lem:finitecenter} applies Lemma \ref{lem:kernelsemidirectreverse} to obtain conditions implying a finite-index subgroup of $N_H(K)/K$ embeds into $\out(G)$. This normaliser-quotient is the basis of the framework described in Section \ref{sec:applications} (see also the discussion after Proposition \ref{prop:Aut}). Lemma \ref{lem:finitecenter} is fundamental to the paper \cite{LoganNonRecursive}. The hypotheses of Lemma \ref{lem:finitecenter} occur when, for example, $H$ is non-elementary hyperbolic, $K\unlhd W$ and $W$ is some torsion-free subgroup of $H$ (that is, $W$ contains no non-trivial elements of finite order).
\begin{lemma}
\label{lem:finitecenter}
If $K\unlhd W\leq N_H(K)$ with $W\cap Z(H)\cap\operatorname{Fix}(\phi)=1$ then $W/K$ embeds into $\out_H^{(V)}(G)$ with index dividing $|C_H(K)|\cdot |N_H(K): W|$.
\end{lemma}

\begin{proof}
Recall that $J:= Z(H)\cap \operatorname{Fix}(\phi)$. To prove that $W/K$ embeds into $\out_H^{(V)}(G)$, note that the map $\overline{\chi_2}$ induces maps $N_H(K)\rightarrow N_H(K)/K\rightarrow N_H(K)/JK$ (the maps are the canonical maps). It is sufficient to prove that if $w\in W\setminus K$ then $wK\not\mapsto JK$ under the second map. Suppose $wK\mapsto JK$ and $wK\neq K$, then there exists some $k\in K$ such that $wk\in J$. However, $wk\in W$ and $W\cap J=1$ so $w=k^{-1}\in K$ a contradiction. Hence, $W/K$ embeds into $\out_H^{(V)}(G)$.

To prove that the index of $W/K$ in $\out_H^{(V)}(G)$ divides $|C_H(K)|\cdot |N_H(K): W|$, first note that, by considering the map $\overline{\chi_2}$, the index of $W/K$ in $\out_H^{(V)}(G)$ clearly divides \[|\out_H^{(V)}(G):\im(\overline{\chi_2})|\cdot|N_H(K):W|.\]
We therefore prove that $|\out_H^{(V)}(G):\im(\overline{\chi_2})|$ divides $|C_H(K)|$. It is routine to construct an isomorphism between $P/(P\cap Q)$ and $PQ/Q$, $Q$ not necessarily normal in $PQ$. Take $P:=C_K$ and $Q:=N_K$. Noting that $\im\left(\overline{\chi_2}\right)=N_K$ and that $C_KN_K=\out_H^{(V)}(G)$, we have that $|\out_H^{(V)}(G):\im\left(\overline{\chi_2}\right)|$ divides $|C_K|$. Then $|C_K|$ divides $|C_H(K)|$ by Lemma \ref{lem:BKisom}, as required.
\end{proof}

\p{Describing \boldmath{$\aut_H^0(G)$}}
Our description of $\out_H^0(G)$ from Theorems~\ref{thm:description} and \ref{thm:trivialcenter} can be used to describe the full pre-image $\aut_H^0(G)$ of $\out_H^0(G)$ in $\aut(G)$. However, this description is not necessarily useful. Lemma \ref{lem:splitting} now gives conditions which allow for a complete description of $\aut_H^0(G)$. Lemma \ref{lem:splitting} is applied in the paper \cite{LoganHNN}. Recall that $N_K=\{\widehat{\alpha}_{(\gamma_b, b^{-1}\phi(b))}: b\in N_H(K)\}$.

\begin{lemma}\label{lem:splitting}
If $A_{(K, \phi)}\leq\inn(H)$ and $C_H(K)=1$ then $\out_H^0(G)\cong N_H(K)/K$ and $\aut_H^0(G)$ can be described as follows.
\begin{align*}
\aut_H^0(G)
&=\inn(G)\rtimes\out_H^0(G)\\
&\cong G\rtimes\frac{N_H(K)}{K}
\end{align*}
\end{lemma}

\begin{proof}
Note that ${\alpha}_{(\gamma_{b}, a)}={\alpha}_{(\gamma_{b}, b^{-1}\phi(b))}{\alpha}_{(1, \phi(b)^{-1}ba)}$, with $\phi(b)^{-1}ba\in C_H(\phi(K))$. Then ${\alpha}_{(1, \phi(b)^{-1}ba)}=1$ as $C_H(K)=1$, so $\out_H^{(V)}(G)=N_K$. As $A_{(K, \phi)}\leq\inn(H)$, Theorem \ref{thm:BassJiangSec85} gives us that $\out_H^0(G)=N_K$.

The representatives $\alpha_{(\gamma_b, b^{-1}\phi(b))}$ of $N_K$ generate a subgroup of $\aut(G)$ which contains non-trivial inner automorphisms. We therefore consider the set $\Phi:=\{\varphi_d: d\in N_H(K)\}$ where $\varphi_d=\alpha_{(\gamma_d, d^{-1}\phi(d))}\gamma_d^{-1}$. Now, $\Phi$ is a subgroup of $\aut_H^0(G)$, and it clearly contains a representative for each coset of $\aut_H^0(G)/\inn(G)$. Hence, $\Phi\inn(G)=\aut_H^0(G)$. Then $\varphi_d\in\inn(G)$ if and only if $d\in K$, if and only if $\varphi_d=1$, so $\Phi\cap\inn(G)=1$. Therefore, $\aut_H^0(G)=\inn(G)\rtimes\Phi$. The result then follows as $\Phi\cong N_H(K)/K$ by Proposition \ref{prop:Aut} while $\inn(G)\cong G$ because $C_H(K)=1$.
\end{proof}


\section{Tractable groups with pathological outer automorphism groups}
\label{sec:applications}
In this section we package the technical results of this paper into a framework for the construction of tractable groups with pathological outer automorphism groups. We then apply this framework to prove Theorem \ref{thm:BWfinitelypresented} from the introduction. Other applications of this framework can be found in \cite{LoganNonRecursive} and \cite{LoganHNN}.

\p{The framework}
In the aforementioned framework the pathological properties of $\out(G)$, where $G=H\ast_{(K, \phi)}$, are inherited from the normaliser-quotient $N_H(K)/K$.
The framework is as follows:
\begin{enumerate}
\item Obtain $K<H$ such that $N_H(K)/K$ has the required properties, e.g. is a stipulated group, is non-residually finite, etc.
\item Apply Theorems~\ref{thm:description}~or~\ref{thm:trivialcenter}, or Lemma \ref{lem:finitecenter} to obtain a description of $\out_H(G)$ in terms of $N_H(K)/K$,
\item Apply Theorem \ref{thm:equality} to obtain $\out_H(G)=\out(G)$, and so obtain a description of $\out(G)$ in terms of $N_H(K)/K$,
\item Apply Lemma \ref{lem:splitting} to obtain $\aut_H(G)=\inn(G)\rtimes\out(G)$, and so obtain a description of $\aut(G)$ in terms of $N_H(K)/K$.
\end{enumerate}

\p{Proof of Theorem \ref{thm:BWfinitelypresented}}
We now prove Theorem \ref{thm:BWfinitelypresented}.

\begin{proof}[Proof of Theorem \ref{thm:BWfinitelypresented}]
Suppose $Q$ is finitely presented with Serre's property FA. Now, there exist torsion-free hyperbolic groups with Serre's property FA \cite{Pride1983some}, and hence there exists a torsion-free hyperbolic group $H$ with Serre's property FA and with a finitely generated non-cyclic normal subgroup $K$ such that $H/K\cong Q$ \cite{belegradek2008rips}. Form $G_Q=\langle H, t; k^t=k, k\in K\rangle$. Note that $G_Q$ is residually finite if (and only if) $Q$ is residually finite \cite[Proposition 2.2]{LoganNonRecursive}.

Firstly note that $\out_H(G_Q)=\out(G_Q)$, which follows from Theorem \ref{thm:equality} because $K$ is normal so conjugacy maximal in $H$, and because $H$ has Serre's property FA. Secondly, note that $N_H(K)/K$ embeds with finite index into $\out_H(G_Q)$, which follows from Theorem \ref{thm:description} because $C_H(K)=1$ as $H$ is torsion-free hyperbolic and $K$ is non-cyclic, and because $\out(H)$ is finite as $H$ is hyperbolic with Serre's property FA \cite{levitt2005automorphisms}.

Hence, $Q\cong H/K=N_H(K)/K\leq_f\out_H(G_Q)=\out(G_Q)$ as required.
\end{proof}

\bibliographystyle{amsalpha}
\bibliography{BibTexBibliography}

\providecommand{\bysame}{\leavevmode\hbox to3em{\hrulefill}\thinspace}
\providecommand{\MR}{\relax\ifhmode\unskip\space\fi MR }
\providecommand{\MRhref}[2]{%
  \href{http://www.ams.org/mathscinet-getitem?mr=#1}{#2}
}
\providecommand{\href}[2]{#2}
\begin{thebibliography}{GHMR00}

\bibitem[ALP]{AtesPride}
F.~Ate\c{s}, A.~D. Logan, and S.J. Pride, \emph{Automata and {Zappa-Sz\`ep}
  products of groups}, in preparation.

\bibitem[Bau63]{Baumslag}
G.~Baumslag, \emph{Automorphism groups of residually finite groups}, J. London
  Math. Soc. \textbf{38} (1963), 117--118. \MR{0146271}

\bibitem[BG03]{braun2003outer}
G.~Braun and R.~G{\"o}bel, \emph{Outer automorphisms of locally finite
  {$p$}-groups}, J. Algebra \textbf{264} (2003), no.~1, 55--67. \MR{1980685}

\bibitem[BJ96]{Bass1996automorphism}
H.~Bass and R.~Jiang, \emph{Automorphism groups of tree actions and of graphs
  of groups}, J. Pure Appl. Algebra \textbf{112} (1996), no.~2, 109--155.
  \MR{1402782}

\bibitem[BO08]{belegradek2008rips}
I.~Belegradek and D.~Osin, \emph{Rips construction and {K}azhdan property
  ({T})}, Groups Geom. Dyn. \textbf{2} (2008), no.~1, 1--12. \MR{2367206}

\bibitem[BT78]{BaumslagTretkoff}
B.~Baumslag and M.~Tretkoff, \emph{Residually finite {HNN} extensions}, Comm.
  Algebra \textbf{6} (1978), no.~2, 179--194. \MR{484178}

\bibitem[BW05]{BumaginWise2005}
I.~Bumagin and D.~T. Wise, \emph{Every group is an outer automorphism group of
  a finitely generated group}, J. Pure Appl. Algebra \textbf{200} (2005),
  no.~1-2, 137--147. \MR{2142354}

\bibitem[CL83]{collins1983automorphisms}
D.~J. Collins and F.~Levin, \emph{Automorphisms and {H}opficity of certain
  {B}aumslag-{S}olitar groups}, Arch. Math. (Basel) \textbf{40} (1983), no.~5,
  385--400. \MR{707725}

\bibitem[DGG01]{droste2001all}
M.~Droste, M.~Giraudet, and R.~G{\"o}bel, \emph{All groups are outer
  automorphism groups of simple groups}, J. London Math. Soc. (2) \textbf{64}
  (2001), no.~3, 565--575. \MR{1865550}

\bibitem[DGP11]{dahmani2011random}
F.~Dahmani, V.~Guirardel, and P.~Przytycki, \emph{Random groups do not split},
  Math. Ann. \textbf{349} (2011), no.~3, 657--673. \MR{2755002}

\bibitem[Far11]{farley2011proof}
D.~Farley, \emph{A proof that {T}hompson's groups have infinitely many relative
  ends}, J. Group Theory \textbf{14} (2011), no.~5, 649--656. \MR{2831963}

\bibitem[FM06]{frigerio2005countable}
R.~Frigerio and B.~Martelli, \emph{Countable groups are mapping class groups of
  hyperbolic 3-manifolds}, Math. Res. Lett. \textbf{13} (2006), no.~5-6,
  897--910. \MR{2280783}

\bibitem[GHMR00]{GHMR}
N.~D. Gilbert, J.~Howie, V.~Metaftsis, and E.~Raptis, \emph{Tree actions of
  automorphism groups}, J. Group Theory \textbf{3} (2000), no.~2, 213--223.
  \MR{1753479}

\bibitem[GP00]{gobel2000outer}
R.~G{\"o}bel and A.~T. Paras, \emph{Outer automorphism groups of metabelian
  groups}, J. Pure Appl. Algebra \textbf{149} (2000), no.~3, 251--266.
  \MR{1762767}

\bibitem[{Kat}15]{kato2015higher}
M.~{Kato}, \emph{{Higher dimensional Thompson groups have {Serre's} property
  {FA}}}, arXiv: 1504.06680 (2015).

\bibitem[Koj88]{kojima1988isometry}
S.~Kojima, \emph{Isometry transformations of hyperbolic {$3$}-manifolds},
  Topology Appl. \textbf{29} (1988), no.~3, 297--307. \MR{953960}

\bibitem[Lev05]{levitt2005automorphisms}
G.~Levitt, \emph{Automorphisms of hyperbolic groups and graphs of groups},
  Geom. Dedicata \textbf{114} (2005), 49--70. \MR{2174093}

\bibitem[Lev07]{levitt2007GBSautomorphism}
\bysame, \emph{On the automorphism group of generalized {Baumslag--Solitar}
  groups}, Geometry \& Topology \textbf{11} (2007), no.~1, 473--515.

\bibitem[Log15]{logan2015outer}
A.~D. Logan, \emph{On the outer automorphism groups of finitely generated,
  residually finite groups}, J. Algebra \textbf{423} (2015), 890--901.
  \MR{3283743}

\bibitem[Log16]{LoganNonRecursive}
\bysame, \emph{On a question of {B}umagin and {W}ise}, New York J. Math.
  \textbf{22} (2016), 865--873. \MR{3548127}

\bibitem[Log17]{LoganHNN}
\bysame, \emph{Every group is the outer automorphism group of an
  {HNN}-extension of a fixed triangle group}, arxiv:1709.06441 (2017).

\bibitem[Log19]{Logan2017Residual}
\bysame, \emph{The residual finiteness of (hyperbolic) automorphism-induced
  {HNN}-extensions}, Comm. Algebra (to appear) (2019).

\bibitem[LS77]{L-S}
R.~C. Lyndon and P.~E. Schupp, \emph{Combinatorial group theory},
  Springer-Verlag, Berlin-New York, 1977, Ergebnisse der Mathematik und ihrer
  Grenzgebiete, Band 89. \MR{0577064}

\bibitem[Mat89]{matumoto1989any}
T.~Matumoto, \emph{Any group is represented by an outerautomorphism group},
  Hiroshima Math. J. \textbf{19} (1989), no.~1, 209--219. \MR{1009671}

\bibitem[Mes72]{meskin1972nonresidually}
S.~Meskin, \emph{Nonresidually finite one-relator groups}, Trans. Amer. Math.
  Soc. \textbf{164} (1972), 105--114. \MR{0285589}

\bibitem[Min09]{minasyan2009groups}
A.~Minasyan, \emph{Groups with finitely many conjugacy classes and their
  automorphisms}, Comment. Math. Helv. \textbf{84} (2009), no.~2, 259--296.
  \MR{2495795}

\bibitem[Pet99]{pettet1999automorphism}
M.~R. Pettet, \emph{The automorphism group of a graph product of groups}, Comm.
  Algebra \textbf{27} (1999), no.~10, 4691--4708. \MR{1709254}

\bibitem[Pri83]{Pride1983some}
Stephen~J. Pride, \emph{Some finitely presented groups of cohomological
  dimension two with property ({FA})}, J. Pure Appl. Algebra \textbf{29}
  (1983), no.~2, 167--168. \MR{707619}

\bibitem[Ser74]{serre1974amalgames}
J.P. Serre, \emph{Amalgames et points fixes}, Proceedings of the {S}econd
  {I}nternational {C}onference on the {T}heory of {G}roups ({A}ustralian {N}at.
  {U}niv., {C}anberra, 1973), Springer, Berlin, 1974, pp.~633--640. Lecture
  Notes in Math., Vol. 372. \MR{0376882}

\bibitem[Ser80]{trees}
\bysame, \emph{Trees}, Springer, 1980.

\bibitem[Shi85]{shirvani1985residually}
M.~Shirvani, \emph{On residually finite {HNN}-extensions}, Arch. Math. (Basel)
  \textbf{44} (1985), no.~2, 110--115. \MR{780256}

\bibitem[Wat82]{yasuo1982propertyT}
Y.~Watatani, \emph{Property {T} of {K}azhdan implies property {FA} of {S}erre},
  Math. Japon. \textbf{27} (1982), no.~1, 97--103. \MR{649023}

\end{thebibliography}
\end{document}